\documentclass[preprint,showpacs,showkeys,preprintnumbers,amsmath,amssymb]{revtex4-2}

\usepackage{graphicx}
\usepackage[T1]{fontenc}
\usepackage{dcolumn}
\usepackage{bm}
\usepackage{epsfig}
\usepackage{xcolor}
\usepackage{framed}
\usepackage{amsmath}
\usepackage{amssymb}
\usepackage{amsthm}
\graphicspath{{Images_Paper/}}
\usepackage{hyperref}
\usepackage[english]{babel}
\usepackage{placeins}
\usepackage{float}
\usepackage{caption}
\usepackage{ulem}
\usepackage[font=footnotesize,labelsep=quad,justification = justified]{caption}
\newtheorem{Theorem}{Theorem}
\newtheorem{Corollary}[Theorem]{Corollary}

\renewcommand{\qedsymbol}{$\blacksquare$}

\begin{document}
\title{On the degree distribution of Haros graphs}
\author{Jorge Calero-Sanz$^{1,2}$}
\email{jorge.calero@urjc.es}
\affiliation{$^1$ Departamento de Matem\'atica Aplicada a la Ingenier\'ia Aeroespacial, ETSI Aeron\'autica y del Espacio, Universidad Polit\'ecnica de Madrid, Madrid, Spain;\\$^2$Signal and Communications Theory and Telematic Systems and Computing, Rey Juan Carlos University, Madrid, Spain.}
\date{\today}

\begin{abstract}
Haros graphs is a graph-theoretical representation of real numbers in the unit interval. The degree distribution of the Haros graphs provides information regarding the topological structure and the associated real number. This article provides a comprehensive demonstration of a conjecture concerning the analytical formulation of the degree distribution. Specifically, a theorem outlines the relationship between Haros graphs, the corresponding continued fraction of its associated real number, and the subsequent symbolic paths in the Farey binary Tree. Moreover, an expression continuous and piece-wise linear in subintervals defined by Farey fractions can be derived from an additional conclusion for the degree distribution of Haros graphs.
\end{abstract}

\maketitle
\newpage

\section{Introduction}
The study of the structure of real numbers has been approached from a variety of perspectives \cite{Hardy, Niqui, Vuillemin, Angell}. The representation by continued fractions and the representation through the Farey Tree are examples of canonical representations \cite{AdamczewskiII, Khinchin, Bonnano_Isola}. Recent graph-theoretical research has provided a new representation of real numbers using Haros graphs \cite{HarosPaper}. These graphs are swayed by the approach of Horizontal Visibility Graphs to the quasiperiodic route \cite{FromTime, FeigenbaumGraphs, AnalyticalFeigenbaum, intermit}. Furthermore, Haros graph are  similar to other structures, such as Farey graphs \cite{Singerman, Kurkofka, Kurkofka2}.  Haros graphs provide a graph description of the unit interval $[0,1]$ establishing a one-to-one correspondence with the well-known Farey sequences $\mathcal{F}_n$, where  $$\mathcal{F}_{n} = \left\lbrace\frac{p}{q} \in [0,1]:\ 0\leq p \leq q \leq n, \ (p,q) = 1 \right\rbrace.$$
The Haros graph set $\mathcal{G}$ is generated recursively from an initial graph (defined as two nodes joined by an edge) and the concatenation graph-operator $\oplus$ (shown in Fig. \ref{Fig_1}). Hence, the set  $\mathcal{G}$ may be represented as a binary tree (see Fig. \ref{Fig_1}). Since $\lim_{n} \mathcal{F}_n =  [0,1] $, the bijection can be extended to the unit interval, where rational numbers are associated with finite Haros graphs and the irrational numbers correspond to infinite Haros graphs.  \\

\begin{figure*}[h!]
\includegraphics[width=1\textwidth]{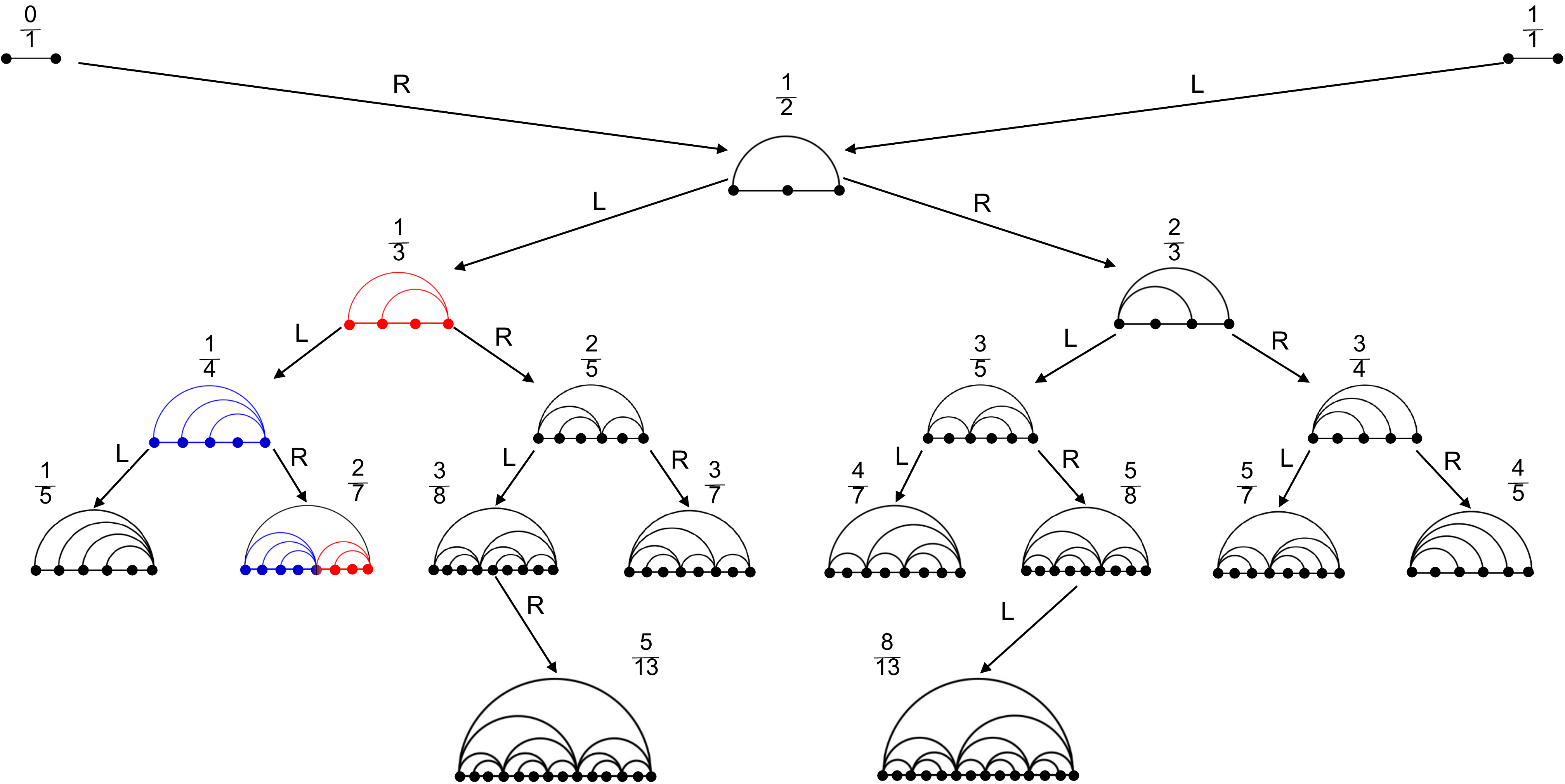}
\caption{Six levels of the Haros graph tree with Haros graphs $G_{p/q}$ associated with the corresponding rational fractions $p/q$ (only two of these are shown at the sixth level due to space constraints). The first level is formed by two copies of the initial graph $G_0$. The graph operator merges the two nearby extreme nodes, adding a connection to the resulting graph that connects the new extreme nodes. On the left, the Haros graph $G_{2/7}$ is generated as a concatenation of $G_{1/4}$ (blue color) and $G_{1/3}$ (red color), i.e., $G_{2/7} = G_{1/4} \oplus G_{1/3}$.}
\label{Fig_1}
\end{figure*}

Consequently, an univocal correspondence $\tau$ exists between real numbers $x \in [0, 1]$ and Haros graphs $\tau(x) = G_x \in \mathcal{G}$. The degree distribution $P(k,x)$ of Haros graphs $G_x$, that is, the probability that a randomly selected node in $G_x$ has degree $k$, is one of the main features investigated in \cite{HarosPaper}.The degree distribution was deemed a fruitfully tool because Haros graphs are uniquely determined by the degree sequence \cite{canonical}, whereas the degree distribution is a marginal distribution of the degree sequence. Indeed, the degree distribution for the three initial values of $k$ confirms:

\begin{equation}
\label{P234}
P(k,x < 1/2)=\left\{
\begin{array}{ll}
x,  & k=2 \\
1 - 2x, & k=3 \\
0, & k=4;
\end{array}%
\right.
P(k,x > 1/2)=\left\{
\begin{array}{ll}
1 - x,  & k=2 \\
2x - 1, & k=3 \\
0, & k=4.
\end{array}%
\right.
\end{equation} \\

In contrast to the initial values, which are related with the real number $x$ associated with $G_x$, the closed form of the degree distribution $P(k,x)$ has only been drawn for degrees $k \geq 5$.  Taking into account the above fact, this paper outlines two theorems to complete the degree distribution expression, based on two distinct approaches. Initially, a complete description of $P(k,x)$ is only provided in terms of the continued fraction of $x$, or, equivalently, in terms of the Haros graph creation process codified along the symbolic binary path. The second result exhibits the properties of $P(k,x)$ as a real value continuous and  piece-wise linear function. \\

This paper is separated into $5$ sections. Section \ref{sec_Preliminaries} gives a brief overview of Haros graphs and its connections to Farey sequences, continued fractions and the Farey binary Tree. The first theorem, provided in section \ref{sec_Tma1}, states the closed form of $P(k,p/q)$ with respect to truncations of the continued fraction of $p/q$. Section \ref{sec_Tma2} rewrites the preceding result using the position of $p/q$ in the Farey binary tree. In addition, Appendix \ref{sec_Appendix}  includes detailed proofs for the assertions in Section \ref{sec_Tma1} and Section \ref{sec_Tma2}.

\section{Preliminaries}
\label{sec_Preliminaries}

The Farey binary tree is a canonical way of represent the set of rational numbers in $[0,1]$ as a binary tree starting with the fractions $0/1, 1/1$ --the elements of Farey sequence $\mathcal{F}_1$-- and creating new irreducible fractions by the mediant sum of two consecutive fractions in $\mathcal{F}_n$:

\begin{equation}
 \frac{p}{q} \oplus \frac{r}{s} = \frac{p + r}{q + s}.
\label{mediantsum}
\end{equation}

The binary tree representation allocates each rational to a level $k$ of the tree, denoted $\ell_k$. For instance, the three first levels consist of:

$$\ell_1=\left\lbrace\frac{0}{1},  \frac{1}{1}\right\rbrace; \ \ell_{2}=\left\lbrace\frac{1}{2}\right\rbrace; \ \ell_{3}=\left\lbrace \frac{1}{3},  \frac{2}{3}\right\rbrace.$$

This representation is closely related with the continued fraction, a powerful technique for representing a real number in the interval $[0,1]$ as
$$x = \cfrac{1}{a_1 + \cfrac{1}{a_2 + \cfrac{1}{a_3 + ...}}} = [a_1, a_2,a_3,...].$$

The relationship has been presented in \cite{Bonnano_Isola}, and it has been establishes that a number with a continued fraction expression $[a_1, a_2,a_3,...]$ has associated a symbolic binary path in the Farey binary tree $L^{a_1} R^{a_2} L^{a_3} ... $, where $L^q$ is interpreted as a sequence of $q$ symbols $L$ (if the symbolic path is finite, the last symbol has an index $a_n -1$). Therefore, since every irrational number has an infinite continued fraction, the irrational numbers are reachable through an infinite path in the Farey binary Tree. Moreover, the continued fraction allows a sequence of rational so-called convergents \cite{Khinchin} defined as:
\begin{eqnarray}
\left\{
\begin{array}{ll}
p_k =  & a_k \cdot p_{k-1} + p_{k-2} \\
q_k =  & a_k \cdot q_{k-1} + q_{k-2}, \\
\end{array}%
\right.
\label{Convergentes_recurrencia}
\end{eqnarray}

with initial values $p_{-2} = 0, q_{-2} = 1, p_{-1} = 1, q_{-1} = 0$, where   $$ \lim_{k \to \infty} \frac{p_k}{q_k} = \lim_{k \to \infty} [a_1, ... , a_k] = [a_1, a_2, \dots] = x.$$

As stated in the preceding section, the Haros graphs set $\mathcal{G}$ provides a graph-based representation of the unit interval $[0,1]$. The primary objective is to reproduce, in a graph scenario, the mediant sum --described in Eq. \ref{mediantsum}-- that was utilised to construct $\mathcal{F}_n$  \cite{Hardy}. The concatenation graph is depicted in Fig. \ref{Fig_1}, and Flanagan et al. \cite{flanagan2019spectral} provide such a  comprehensive definition. Every Haros graph $G$ (except for the initial graph) is therefore described as $G = G_L \oplus G_R$, where $G_L, G_R$ are also Haros graphs. However, just as mediant sum only takes to two nearby fractions in Farey sequences, Haros graphs only can be concatenated if $G_L, G_R$ are also adjacent. \\

In order to analyse the topological structure of Haros graph, the probability distribution degree $P(k,x)$ proves to be a useful instrument. Eq. \ref{P234} identifies the three first values of $k$, although by construction $P(k,0) = P(k,1) = 0$. In addition, for degrees $k \geq 5$, Theorem 2 in \cite{HarosPaper} determines that the degree values for which $P(k,x) = 0$ rely on the symbol repetition -- $RR$ or $LL$ -- in the symbolic path of the Haros graph Tree to reach $G_x$. Also, the same research conjectures regarding the closed form of $P(k,x)$. The aim is to provide a formal proof of this claim. \\

Prior to this, a brief explanation of the emergence of degrees $k \geq 5$ is provided: emerging of degrees $k \geq 5$ occurs if there is a change of symbol $L \to R$ or $R \to L$ in the path of the Haros graph Tree. Consequently, the degree emerged, or not, is related to the level at which this  symbolic change occurs. Suppose that we have covered the path $L^{a_1}R^{a_2} ... L^{a_{k-1}}$. Next, the downstream of $R^{a_k}$ generates a new degree, which had previously appeared as a boundary node before to the shift in direction (the node resulting of the identification of extreme nodes). Specifically, in the first downstream $R$, the degree appears in the merging node, and the number of nodes with this degree increases by one with each descent. Therefore, when we reach $R^{a_k - 1}$, there will be $0 + (a_k - 1) = a_k - 1$ nodes of that degree. In addition, the Haros graph achieved is associated with $p_k /q_k = [a_1, ..., a_k]$. \\

Now, in order to reach the Haros graph associated with $[a_1, ..., a_k, a_{k+1}]$, we must descend to $R$ and then $a_{k+1} - 1$ times to $L$. If there were $a_k - 1$ nodes of that degree previously, a new descent will result in $a_k$ nodes. Then, performing a descent $L^{a_{k+1} - 1}$ requires concatenating the Haros graph reached by $R^{a_k}$ with $a_{k+1} - 1$ copies of the Haros graph reached by $R^{a_k - 1}$, so that the resultant Haros graph has $(a_k - 1)\cdot (a_{k+1} - 1) + a_k = (a_k - 1)\cdot (a_{k+1}) + 1$ nodes of that degree. \\

In other words, the emerging of the degree according to the recursive equation $q_n = q_{n-2} + a_n \cdot q_{n-1}$, where the terms $a_{i}$ correspond to the continued fraction $[a_k - 1, a_{k+1}, ..., a_m]$, with initial conditions $q_{-1} = 0, q_0 = 1$. Hence, this recursive equation converges to the denominator of the continued fraction $[a_k - 1, a_{k+1}, ..., a_m]$. \\

\section{Degree distribution of rational Haros graphs \texorpdfstring{$G_{p/q}$}{a} based on continued fraction \texorpdfstring{$p/q = [a_1, ... , a_m]$}{b}}
\label{sec_Tma1}

The first presented result provides an explicit description of the degree distribution related with truncations of the continued fraction of $p/q$, the rational associated with Haros graph $G_{p/q}$: 
\begin{Theorem}
Let $p/q \in [0,1/2]$, with $p/q = [a_1, ... , a_m]$. Then, the degree distribution $P(k,p/q)$ of the Haros graph $G_{p/q}$ is:
\begin{equation}
P\left(k,\frac{p}{q}\right)=\left\{
\begin{array}{ll}
p/q,  & k=2 \\
(q-2p)/q, & k=3 \\
0, & k=4 \\
s^{(l)}/q, & \textrm{for values } k = \sum_{i = 1}^{l} a_i + 3, \textrm{with } \forall l = 1, ..., m - 1, \\
& \textrm{where }r^{(l)}/s^{(l)} = [a_{l+1} -1 , ... , a_m] \\
1/q, & \textrm{if } k = \sum_{i = 1}^{m} a_i + 2\\
0, & \textrm{otherwhise.}
\end{array}%
\right.
\label{GH_Pkxeq}
\end{equation}
\label{GH_tma1}
\end{Theorem}

\begin{proof}
See Section \ref{subsection_proof_tma1} for a complete proof.
\end{proof}

The theorem has several consequences: first, it unveils the whole expression for $P(k,x)$ providing a wide amount of topological information. Moreover, as stated in the previous section, the values $k \geq 5$ are related to the continued fraction and, consequently, with the symbolic path reached in the Haros graph Tree. In addition,  the result reduces the computational cost  for obtaining the degree distribution $P(k,p/q)$ of the Haros graph $G_{p/q}$. Initially, observe that the denominator of the $n$-th convergent $q_n$ of $p/q$ verifies $q_n \geq \phi^{n-1}$, where $\phi$ is the Golden number \cite{Tsigaridas}. Hence, as the Haros graph $G_{p_n /q_n}$ has $q_n + 1$ nodes, its growth is exponential, but the theorem \ref{GH_tma1} uses only continued fractions \cite{Vuillemin, Welch}. \\

Let us illustrates an example: consider the Haros graph $G_{10/23}$, where $10/23 = [a_1, a_2, a_3] =  [2,3,3]$. Hence, the symbolic path in the Haros graph tree is $L^2R^3L^2$ (see Fig. \ref{Fig_2} for an illustration). Numerically, its degree distribution is as follows:
 
 \begin{equation}
P(k,10/23)=\left\{
\begin{array}{ll}
\frac{10}{23},  & k=2 \\
\frac{3}{23}, & k=3 \\
0, & k=4 \\
\frac{7}{23}, & k=5 \\
\frac{2}{23}, & k=8 \\
\frac{1}{23}, & k=10 \\
0, & \text{ otherwise.}
\end{array}%
\right.
\label{GH_P1023_eq}
\end{equation}
 
\begin{center}
\begin{figure*}[h!]
\includegraphics[width=0.9\textwidth, trim=0cm 3cm 0cm 3cm,clip=true]{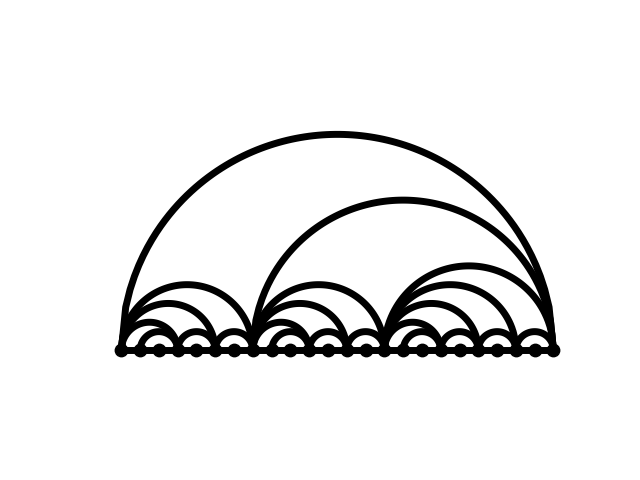}
\caption{Haros graph $G_{10/23}$. As $10/23 = [2,3,3]$, it is clear that the binary symbolic path to reach this Haros graph is $LLRRRLL = L^2 R^3 L^{3-1}$. According to the boundary node convention, the extreme nodes are identified as a single boundary node, while the total of degrees is maintained. Then, the degree sequence is $[3,2,5,2,5,2,8,3,2,5,2,5,2,8,3,2,5,2,5,2,5,2, 5+5 =10]$.}
\label{Fig_2}
\end{figure*}
\end{center}

Then, we can verify that there are $m_{5, 10/23} = 7$ nodes of degree $k = 5$, and that this number corresponds with the denominator of the continued fraction
$$[a_2 - 1, a_3 ] = [2,3] = \frac{3}{7};$$
Also, there are $m_{8, 10/23} = 2$ nodes of degree $k = 8$, which is the denominator of the continued fraction.
$$[a_3 - 1] = [2] = \frac{1}{2}.$$
Lastly, $k = 10$ corresponds with the degree of the boundary node, therefore it only appears once.

\section{Degree distribution of rational Haros graphs \texorpdfstring{$G_{p/q}$}{a} related to the rational number \texorpdfstring{$p/q$}{a}}
\label{sec_Tma2}

With the theorem \ref{GH_tma1}, we are able to provide a proof for the conjecture presented in \cite{HarosPaper}. Contrary to the previous finding, the subsequent theorem outcomes in a formulation of the degree distribution relating with the number $x$, but the computation of $P(k,x)$ for every $k$ depends on the location of $x$ in the sub-intervals given by the levels in the Farey binary tree $\ell_k$.  \\

\begin{Theorem}
Let $p/q \in [0,1/2]$ and the associated Haros graph $G_{p/q}$. Let consider  $\ell_{k-3} = \left\{ \frac{a_i}{b_i} \right\}_{i = 1}^{2^{k-5}}$, and $\ell_{k-2} = \left\{ \frac{p_j}{q_j} \right\}_{j = 1}^{2^{k-4}}$, clearly we have that $\forall i = 1, ..., 2^{k-5}$:

\begin{equation}
 \frac{p_{2i-1}}{q_{2i-1}} < \frac{a_i}{b_i} <\frac{p_{2i}}{q_{2i}}
\end{equation}

Therefore, the degree distribution of Haros graph $G_{p/q}$ for degrees $k \geq 5$ is: 
\begin{equation}
P\left( k,\frac{p}{q} \right)=\left\{
\begin{array}{ll}
q_{2i-1} \cdot (p/q) - p_{2i-1}, \  \textnormal{if }\  p/q \in\left(\frac{p_{2i-1}}{q_{2i-1}} ,\frac{a_i}{b_i}\right), \forall i = 1, ... , 2^{k-5},  \\ 
-q_{2i} \cdot (p/q) + p_{2i}, \ \textnormal{if  }\ p/q \in \left( \frac{a_i}{b_i} ,\frac{p_{2i}}{q_{2i}} \right), \forall i = 1, ... , 2^{k-5},\\  
1/q_j,\ \textnormal{si }\ \frac{p}{q} =\frac{p_j}{q_j} \in \ell_{k-2},\\ 
 0,\ \textnormal{otherwise}.
\end{array}%
\right.
\label{GH_Pkxgen}
\end{equation}
\label{GH_tma2}
\end{Theorem} 

\begin{proof}
See Section \ref{subsection_proof_tma2} for a complete proof.
\end{proof}

The theorem can be extended from rational numbers to all real numbers. Fig. \ref{Fig_3} depicts a numerical computation of $P(k,x)$ for the first values of $k \geq 5$ verifying the statement of Theorem \ref{GH_tma2}. Moreover, this new formulation allows following statement to emphasise some aspects of the degree distribution $P(k,x)$ as a real function over the variable $x$:

\begin{Corollary}
The degree distribution $P(k,x)$ over the variable $x$ is a piece-wise linear and continuous function, with the exception of  measure null set.
\end{Corollary}

Now, let us illustrate the theorem by applying the result to the case when $k = 5$. Then, it is a simple matter to confirm that if $p/q \in (1/3, 1/2)$, then:

$$P\left(k= 5, \frac{p}{q} \right) = 3 \cdot \frac{p}{q} - 1 = \frac{3p - q}{q}.$$

In comparison, the continued fractions of rational numbers  $p/q \in (1/3, 1/2)$ start with the term $a_1 = 2$. In virtue of theorem \ref{GH_tma1}, the degree $k = a_1 + 3 = 5$ would have a frequency of $s^{(1)}/q$, where $s^{(1)}$ is the denominator of $[a_2 - 1, a_3, ..., a_m]$. Let us examine how the two expressions coincide:

\begin{eqnarray*}
\frac{p}{q} = \cfrac{1}{2 + \cfrac{1}{a_2 + \ddots }} \Rightarrow \frac{q}{p} - 2 = \frac{q - 2p}{p} = \cfrac{1}{a_2 + \cfrac{1}{a_3 + \ddots }} \Rightarrow  \\
\frac{p}{q - 2p} -1 = \frac{3p - q}{q - 2p} = (a_2 - 1) +  \cfrac{1}{a_3 + \cfrac{1}{a_4 + \ddots }} \Rightarrow  \\
\frac{q - 2p}{3p - q}  = \cfrac{1}{(a_2 - 1) +  \cfrac{1}{a_3 + \cfrac{1}{a_4 + \ddots }}}. \\
\end{eqnarray*}

Hence, $3p-q$ is the denominator of the continued fraction $[a_2 - 1, a_3, ... , a_m]$ according to the theorem \ref{GH_tma1}. This finding may be generalisable to all degree value $k$, requiring the partition of $[0,1]$ by the levels $\ell_{k-3}$ and $\ell_{k-2}$.

\begin{figure}[H]
\includegraphics[width=1\textwidth]{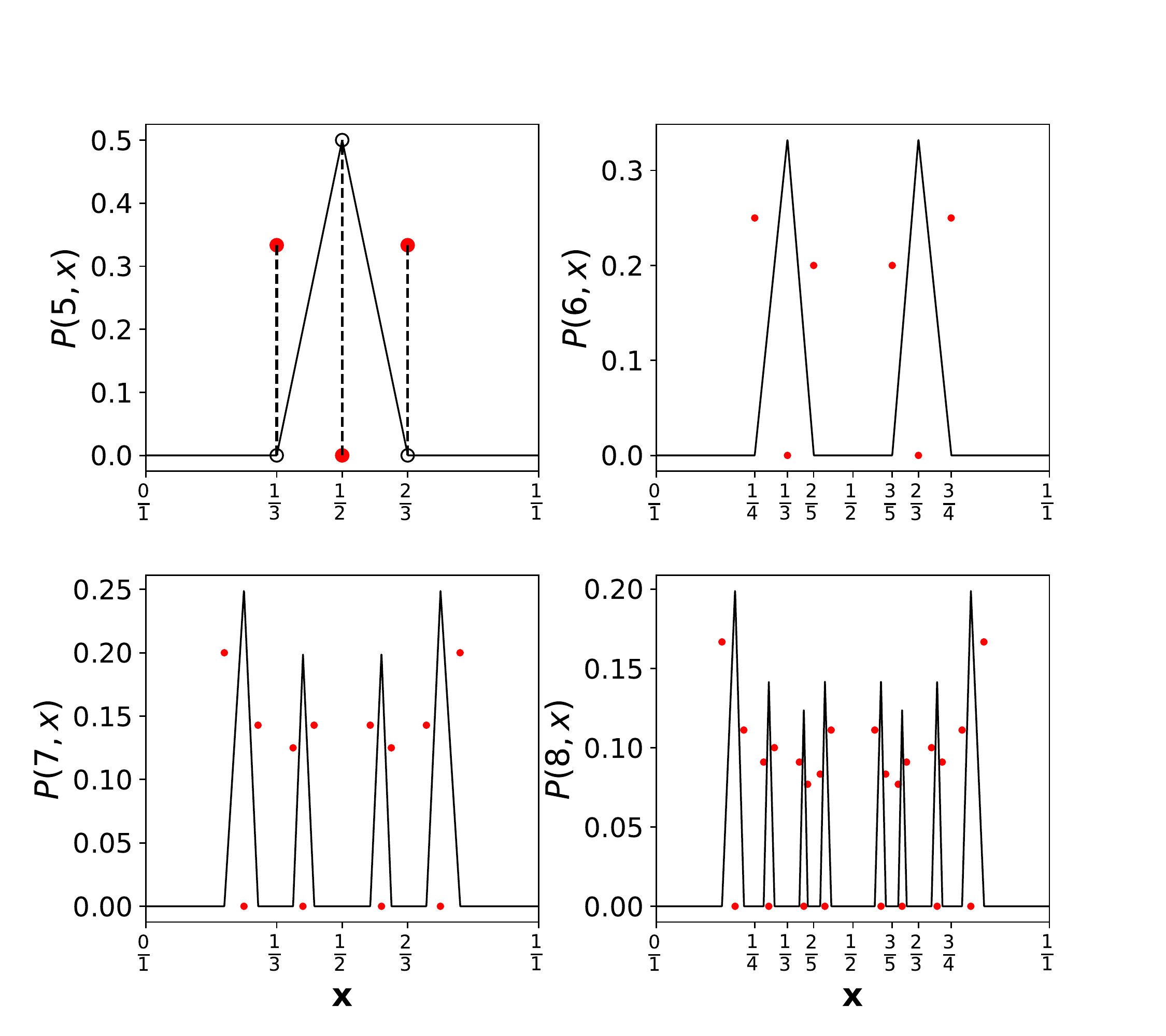}
\caption{Numerical computation of Degree distribution $P(k,x)$ as a function of $x$ for $k = 5, 6, 7, 8$, and for all Haros graphs $G_x$ with $x \in \mathcal{F}_{1000}$. The red points represent removable discontinuities whereas the solid black lines shows the piece-wise linear behaviour. Due to lack of space, only the left upper panel of $P(5,x)$ accurately shows that the red points represent the Haros graph located in levels $\ell_{\kappa}$, for $\kappa = 5 - 3 = 2$, i.e., the Haros graph $G_{1/2}$ (without nodes of degree $k = 5$), and for $\kappa = 5 - 2 = 3$, i.e., the Haros graphs $G_{1/3},G_{2/3}$, where the degree $k = 5$ is located at the boundary node.}
\label{Fig_3}
\end{figure}

\section{Appendix}
\label{sec_Appendix}

\subsection{Proof of Theorem \ref{GH_tma1}}
\label{subsection_proof_tma1}

\begin{proof}
The values $k = 2, 3, 4$ are calculated in \ref{P234}, while the value $k = \sum_{i = 1}^{m} a_i + 2$ corresponds to the degree of the boundary node, which has frequency $1$ according to the the results of Appendix B. of the \cite{HarosPaper}. \\

Let us consider the Haros graph $G_{p/q}$, with $p/q = [a_1, ... , a_m]$. Thus, the symbolic path in the Haros graph Tree is  $L^{a_1}R^{a_2} ... X^{a_m - 1}$, with $X = L$ or $X = R$ depending on whether $m$ is odd or even, respectively. In virtue of the Theorem 2 presented in \cite{HarosPaper}, it can be determined that non-zero values of the degree distribution $P(k, p/q)$ correspond to the degrees $k = \sum_{i = 1}^{l} a_i + 3$ for the values $l = 1, 2, ..., m-1$. It remains, however, to determine its precise value.  To accomplish this, we shall conduct an induction exercise on the levels of the Haros graph Tree $\ell_n$: \\

At level $\ell_4$, the first Haros graph with degree $k \geq 5$ is associated to $2/5 =[a_1, a_2] = [2,2]$ proving that:

$$P\left(a_1 + 3, \frac{2}{5} \right) = P\left(5, \frac{2}{5} \right) = \frac{1}{5},$$

the value of the denominator of the continued fraction $[a_2 - 1]=[2-1] = [1] = 1/1$. This level also includes the Haros graph $G_{1/4}$, which fulfils the theorem because it contains only the boundary node with degree $k \geq 5$. \\

Assume the result holds for all levels of the Haros graph tree $\ell_{\kappa}$, with $\kappa \leq n$. Let us verify that it is satisfied at level $\kappa = n + 1$. To accomplish this, consider an arbitrary Haros graph $G_{p/q}$ at level $n$, where $p/q = [a_1, ... , a_m]$. Henceforth, we shall demonstrate  that both its left and right descendants satisfy the Eq. \ref{GH_Pkxeq}. \\

Let us suppose, without loss of generality, that the last descents that result until we reach $G_{p/q}$ are left descents $L$, as seen in Fig. \ref{GH_fig7}. Also, Cvitanovic et al. \cite{Cvitanovic} show that if $p/q = [a_1, ... , a_m]$, then its left descendant is expressed in continued fraction as $[a_1, ... , a_m + 1]$; whereas the right descendant has a continued fraction expression $[a_1, ... , a_m- 1, 2]$. \\

\begin{figure}[htpb]
\begin{center}
\includegraphics[width=1\textwidth, trim=0cm 2.5cm 0cm 30cm,clip=true]{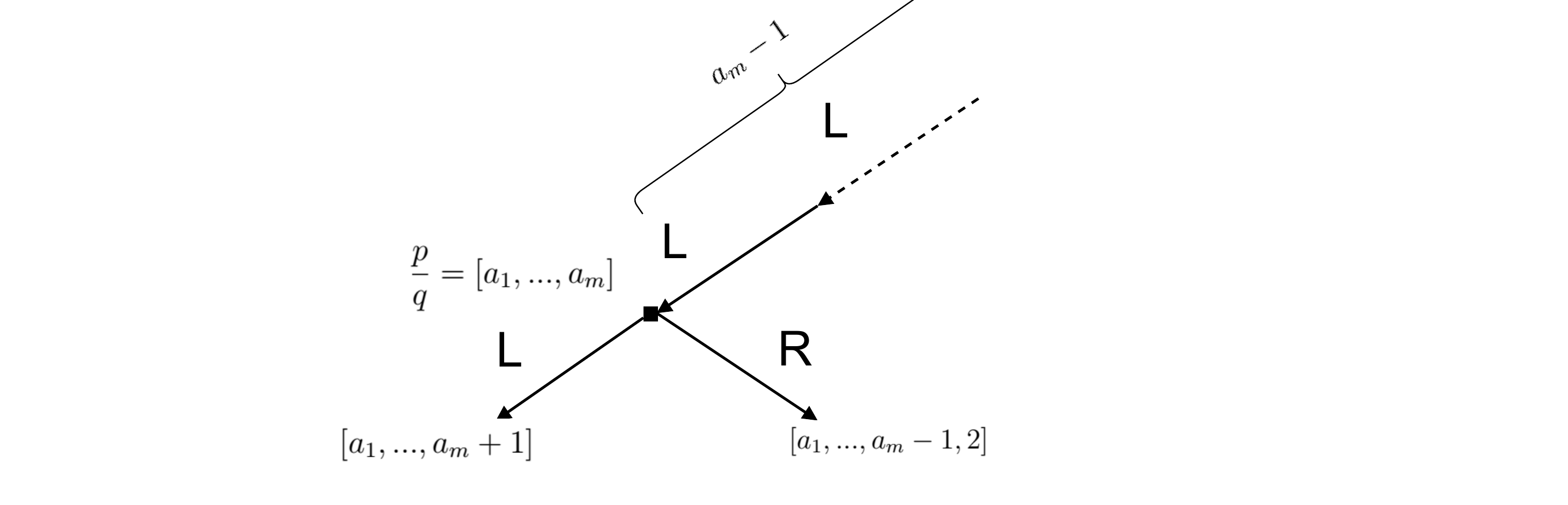}
\caption{Diagram of the descendants of the Haros graph $G_{p/q}$, with $p/q = [a_1, ... , a_m]$. Let us suppose that the Haros graph $G_{p/q}$ contains a symbolic path that terminates in  $a_m - 1$ descents to the left $L$ (the result is analogous for descents to  right $R$). The continued fraction for the left descendant of the Haros graph is $[a_1, ... , a_m + 1]$, whereas the right descendant has a continued fraction $[a_1, ... , a_m- 1, 2]$.}
\label{GH_fig7}
\end{center}
\end{figure} 
  
Let us start by dealing with the left descendant. This case is depicted in Fig. \ref{GH_fig8}, which illustrates how this descendant is formed by concatenating the convergent of order $m - 1$ of $p/q$ with $p/q$ itself. By setting a value $l \in \left\lbrace 1 , 2 ... , m - 2 \right\rbrace$, therefore we will be in the degree $k_l = \sum_{i = 1}^{l} a_i + 3$. Now, it is clear that the number of nodes with degree $k_l$ in the left descendant is equal to the total of the number of nodes with that degree in its two ascendants $G_{p_{m-1}/q_{m-1}}$ and $G_{p/q}$.\\

\begin{figure}[htpb]
\begin{center}
\includegraphics[width=0.9\textwidth, trim=0cm 3cm 0cm 26cm,clip=true]{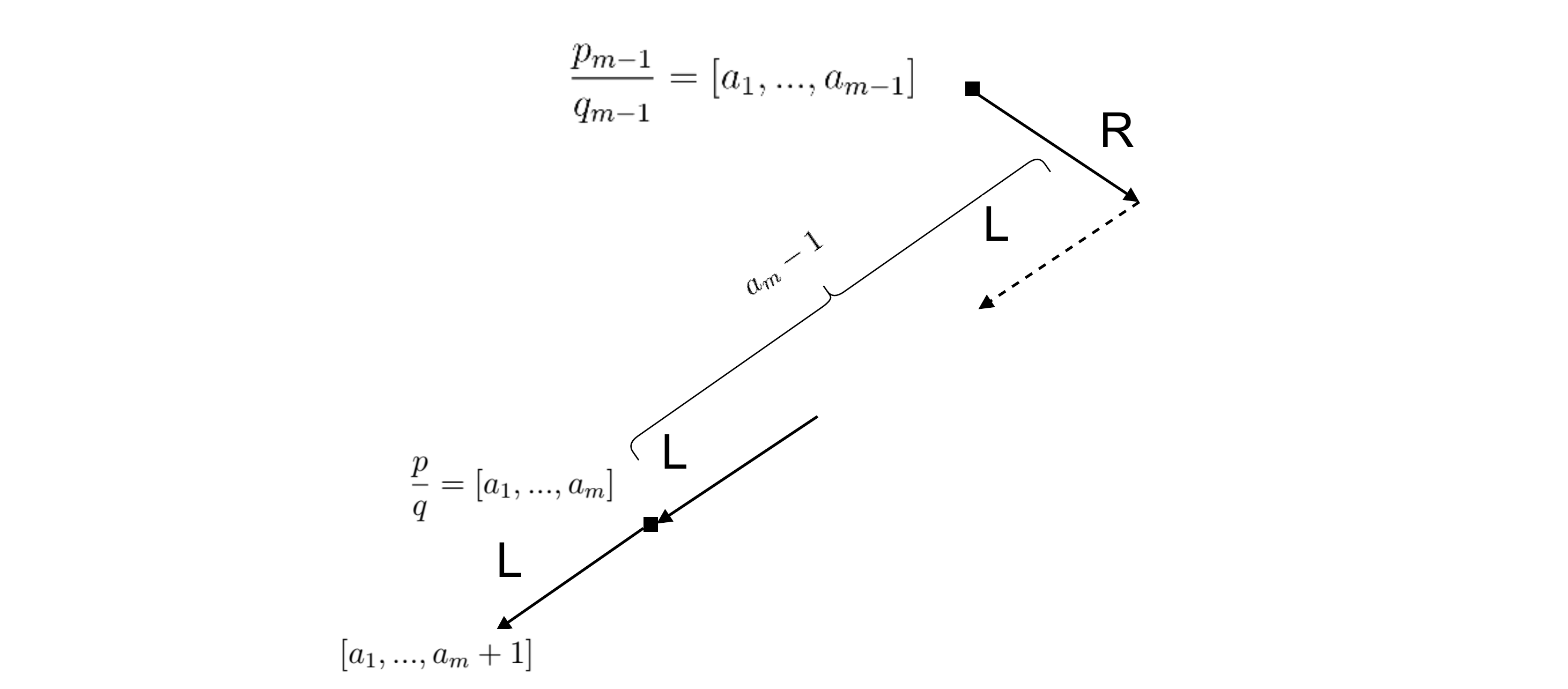}
\caption{Diagram of the construction for the left descendant of $G_{p/q}$, where $p/q = [a_1, ... , a_m]$. The left descendant is expressed as $[a_1, ... , a_m+1]$ and is obtained by concatenating $G_{p_{m-1}/q_{m-1}}$ and $G_{p/q}$, where $p_{m-1}/q_{m-1}$ is the convergent of order $m-1$ of $p/q$.}
\label{GH_fig8}
\end{center}
\end{figure}

By induction hypothesis, the denominators of the continued fractions are the number of nodes of degree $k_l$ of its ascendants, denoted by $s_{p/q}^{(l,m)}$ and $s_{p/q}^{(l,m-1)}$, respectively. Hence, we have: 

\begin{equation}
\frac{r_{p/q}^{(l,m)}}{s_{p/q}^{(l,m)}} :=  [a_{l+1} -1 , ... , a_m],
\label{GH_eq1}
\end{equation}

and
 
\begin{equation}
\frac{r_{p/q}^{(l,m-1)}}{s_{p/q}^{(l,m-1)}} :=   [a_{l+1} -1, ... , a_{m-1}].
\label{GH_eq2}
\end{equation}

This notation also reflects that the term of Eq. \ref{GH_eq2} is the previous convergent of Eq. \ref{GH_eq1}. The recursivity of the continued fractions  (shown in Eq. \ref{Convergentes_recurrencia}) entails that we have:

\begin{equation}
s_{p/q}^{(l,m)} = s_{p/q}^{(l,m-2)} + a_m \cdot s_{p/q}^{(l,m-1)}.
\label{GH_eq3}
\end{equation}

Denoting by $s_{L}^{(l)}$ the number of nodes with degree $k_l$ in the left descendant, we get the conclusion that:

\begin{align}
s_{L}^{(l)} &=& s_{p/q}^{(l,m)} + s_{p/q}^{(l,m-1)} = \left(s_{p/q}^{(l,m-2)} + a_m \cdot s_{p/q}^{(l,m-1)}\right) + s_{p/q}^{(l,m-1)} \nonumber \\
&=& s_{p/q}^{(l,m-2)} + (a_m + 1) \cdot s_{p/q}^{(l,m-1)},
\label{GH_eq4}
\end{align}

where the right-hand side of Eq. \ref{GH_eq4} is the denominator of the continued fraction $[a_{l+1} - 1, ..., a_{m} + 1]$, as we aimed to demonstrate. \\

The value $l = m - 1$ is required to complete the study of the left descendant; therefore, it is necessary to check the result for the degree $k_{m-1} = \sum_{i = 1}^{m-1} a_i + 3$. The reason for separate consideration is that $G_{p_{m-1}/q_{m-1}}$ does not contain the degree $k_{m-1}$. Nonetheless, this degree first appears at the level  $\sum_{i = 1}^{m-1} a_i + 1$ of the Haros graph tree, i.e., as the boundary node of the right descendant of $G_{p_{m-1}/q_{m-1}}$. The number of nodes with that degree will increase by one with each successive left descent until the Haros graph $G_{p/q}$ is reached.  Moreover, the induction hypothesis applied to $p/q$ determines that there are $s_{p/q}^{(m-1,m)}$ nodes of this degree, i.e., the denominator of

$$[a_m - 1] = \frac{1}{a_m - 1}.$$

Therefore, there will be $s_{p/q}^{(m-1,m)} + 1$ nodes, or equivalently, there will be $a_m$ nodes of this degree. This number is the denominator of the truncated continued fraction $[(a_m + 1) - 1] = [a_m]$, thus finishing the study of left descent.  \\ 

Let us now turn our attention to the right descendant, depicted in Fig. \ref{GH_fig9}. In this instance, the concatenation occurs between the Haros graph $G_{p/q}$ and its right ancestor, i.e., the Haros graph associated to the continued fraction $[a_1, ... , a_m -1]$.   \\

\begin{figure}[H]
\begin{center}
\includegraphics[width=1\textwidth, trim=0cm 3cm 0cm 30cm,clip=true]{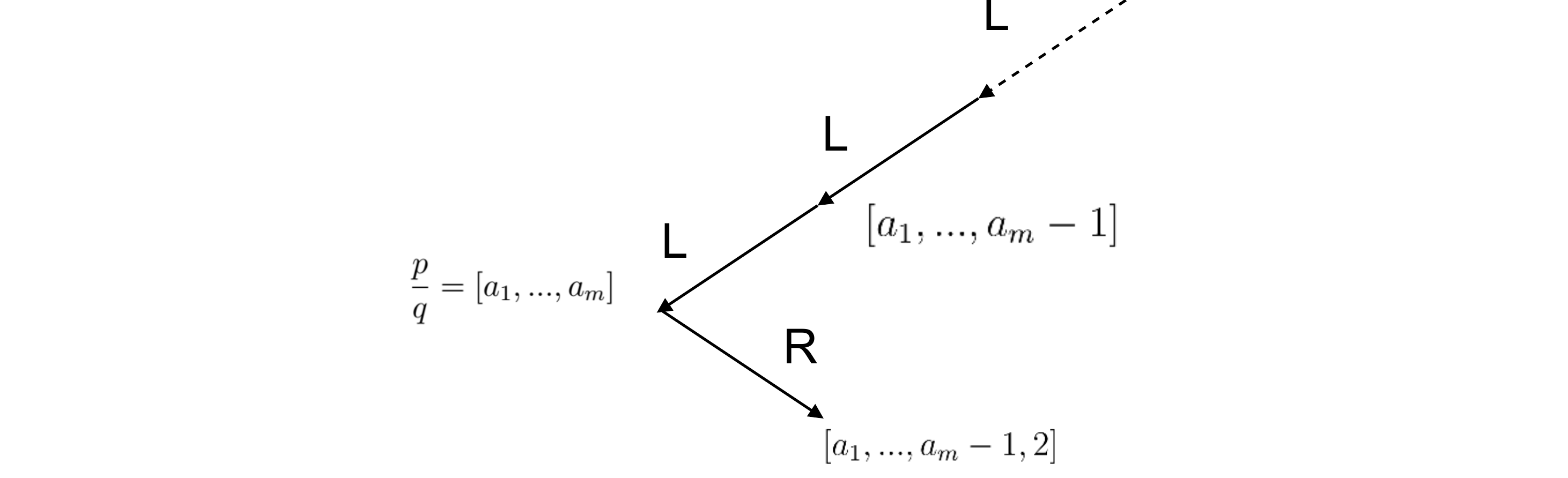}
\caption{Diagram of the construction for the right descendant of $G_{p/q}$, where $p/q = [a_1, ... , a_m]$. The right descendant is expressed as $[a_1, ... , a_m+1]$ and is obtained by concatenating the Haros graphs $G_{p_{m-1}/q_{m-1}}$ and $G_{p/q}$, where $p_{m-1}/q_{m-1}$ is the $m-1$-th convergent of $p/q$.}
\label{GH_fig9}
\end{center}
\end{figure}

Let us consider the degrees $k_l = \sum_{i = 1}^{l} a_i + 3$ with  $l \in \left\lbrace 1 , 2, ... , m - 1 \right\rbrace$. The right descendant is obtained by concatenating $G_{p/q} \oplus G_{a/b}$, where $\frac{a}{b} = [a_1, ... , a_m - 1]$. Therefore, we must represent the truncations of $p/q$ and $a/b$, as well as their convergents, which are indicated as follows:

\begin{equation}
\frac{r_{p/q}^{(l,m)}}{s_{p/q}^{(l,m)}} =  [a_{l+1} -1 , ... ,a_{m-1}, a_m],
\label{GH_eq5}
\end{equation}

and
 
\begin{equation}
\frac{r_{a/b}^{(l,m)}}{s_{a/b}^{(l,m)}} =   [a_{l+1} -1, ... ,a_{m-1}, a_{m} - 1].
\label{GH_eq6}
\end{equation}

It is easy to show that the continued fractions in Eq. \ref{GH_eq5} and Eq. \ref{GH_eq6} match in all the terms with the exception of the last one; hence, all the denominators of the convergent with $t \leq m - 1$, verify that $s_{p/q}^{(l,t)} = s_{a/b}^{(l,t)}$. Using the induction hypothesis, the recursive equations \ref{Convergentes_recurrencia} for the convergent, and denoting by $s_{R}^{(l)}$ the number of nodes of degree $k_l$ in the right descendant, we have:
\begin{eqnarray}
s_{R}^{(l)} & = & s_{p/q}^{(l,m)} + s_{a/b}^{(l,m)} \nonumber \\
       & = & (s_{p/q}^{(l,m-2)} + (a_m) \cdot s_{p/q}^{(l,m-1)}) + (s_{a/b}^{(l,m-2)} + (a_m - 1) \cdot s_{a/b}^{(l,m-1)}) \nonumber \\
       & = & (s_{a/b}^{(l,m-2)} + (a_m) \cdot s_{a/b}^{(l,m-1)}) + (s_{a/b}^{(l,m-2)} + (a_m - 1) \cdot s_{a/b}^{(l,m-1)}) \nonumber \\
       & = & 2\cdot s_{a/b}^{(l,m-2)} + 2\cdot (a_m - 1) \cdot s_{a/b}^{(l,m-1)} +  s_{p/q}^{(l,m-1)} \nonumber \\
       & = & 2 ( s_{a/b}^{(l,m-2)} + \cdot (a_m - 1) \cdot s_{a/b}^{(l,m-1)}) +  s_{a/b}^{(l,m-1)} \nonumber \\
       & = & 2 \cdot s_{a/b}^{(l,m)} +  s_{a/b}^{(l,m-1)},
\label{GH_eq7}
\end{eqnarray}

where the right-hand side of Eq. \ref{GH_eq7} is the denominator of the truncated continued fraction $[a_{l+1} -1, ... ,a_{m-1}, a_{m} - 1,2]$. Finally, the right descent has a single node with degree

$$k = \sum_{i = 1}^{m-1} a_i + (a_m - 1) + 3 = \sum_{i = 1}^{m} a_i + 2.$$ 

This is the merging node obtained by the concatenating the extreme nodes of $p/q$ and $a/b$, in agreement with the denominator of the last truncated right descendant, which is $[2 - 1] = [1] = 1/1$, thereby ending the proof.
\end{proof}

\subsection{Proof of Theorem \ref{GH_tma2}}
\label{subsection_proof_tma2}

In order to demonstrate the result, we must introduce the continuants, a recursively defined polynomials stated by Euler, where:

\begin{eqnarray}
K_0 & = & 1;  \nonumber \\
K_1(x_1) & = & x_1;   \nonumber \\
K_n(x_1, ..., x_n) & = & x_n \cdot K_{n-1}(x_1, ... , x_{n-1}) + K_{n-2}(x_1, ... , x_{n-2}), \;  \text{ for } \; n \geq 2.
\label{GH_cont}
\end{eqnarray}
The continuants allow express the continued fractions as

\begin{equation}
\frac{p}{q} = [a_1, ... , a_n] = \frac{K_{n-1}(a_2, ... , a_n)}{K_n(a_1, ... , a_n)}  .
\label{GH_contpq}
\end{equation}

In addition, two properties of the continuants must be introduced. The first generalises the definition \ref{GH_cont} $\forall m$, with $1 \leq m < n$ as follows: 

\begin{align}
K_n(x_1, ..., x_n) &=& K_m(x_1, ..., x_m) \cdot K_{n-m}(x_{m+1}, ..., x_n)  \nonumber \\
& + &K_{m-1}(x_1, ..., x_{m-1})\cdot  K_{n-m-1}(x_{m+2}, ..., x_n). 
\label{GH_contgen}
\end{align}

The second equality was established by Muir and Metzler in \cite{Muir}:

\begin{equation}
K_n(x_1, ..., x_n) \cdot K_{n-2}(x_2, ..., x_{n-1}) - K_{n-1}(x_1, ..., x_{n-1}) \cdot K_{n-1}(x_2, ..., x_{n}) = (-1)^n .
\label{GH_contdet}
\end{equation} 

\begin{proof}
The intervals where the degree distribution $P(k,p/q) > 0$ are determined by the elements $a_i/b_i \in \ell_{k-3}$, i.e., the Farey fractions of a certain level in the Farey tree, and their descendants located at a lower level $p_{2i - 1}/q_{2i-1} \in  \ell_{k-2}$. Let us first establish that 

$$\frac{a_i}{b_i} = [\alpha_1, ... , \alpha_r],$$

\noindent then, if $r$ is even, we obtain
\begin{equation}
\frac{p_{2i - 1}}{q_{2i-1}} = [\alpha_1, ... , \alpha_r + 1]; \; \frac{p_{2i}}{q_{2i}} = [\alpha_1, ... , \alpha_r -1, 2],
\label{GH_desc_impar}
\end{equation}
whereas for $r$ odd we have:
\begin{equation}
\frac{p_{2i}}{q_{2i}} = [\alpha_1, ... , \alpha_r + 1]; \; \frac{p_{2i-1}}{q_{2i-1}} = [\alpha_1, ... , \alpha_r -1, 2].
\label{GH_desc_par}
\end{equation}
 \\
 Let us suppose, without loss of generality, the case $r$ odd, and consider a rational number 
$$
\frac{p}{q} \in \left(\frac{p_{2i-1}}{q_{2i-1}} = [\alpha_1, ... , \alpha_r + 1] ,\frac{a_i}{b_i} = [\alpha_1, ... , \alpha_r] \right).
$$
Then, the first $r$ terms of the continued fraction are determined as follows:

$$\frac{p}{q} = [\alpha_1, ... , \alpha_r, b_{r+1}, ... , b_m].$$

Furthermore, if  $a_i/b_i = [\alpha_1, ... , \alpha_r] \in \ell_{k-3}$, then $k = \sum_{i= 1}^{r} \alpha_i + 3$. Let us verify that the numerator of the expression $q_{2i-1} \cdot (p/q) - p_{2i-1}$, stated in theorem \ref{GH_tma2} and the term $s^{(r)}$ stated in theorem \ref{GH_tma1}, have the same value for the degree $k = \sum_{i= 1}^{r} \alpha_i  + 3$. Applying Eq. \ref{GH_contpq} and Eq. \ref{GH_contgen} to  $p/q$ and $p_{2i-1}/q_{2i-1}$, we get:

\begin{eqnarray}
q_{2i-1} \cdot p  - q \cdot p_{2i-1} &=& K(\alpha_1, ... , \alpha_r +1) \cdot K(\alpha_2, ... , \alpha_r, b_{r+1}, ..., b_m) \nonumber \\
& & -K(\alpha_1, ... , \alpha_r, b_{r+1}, ..., b_m) \cdot K(\alpha_2, ... , \alpha_r + 1) \nonumber \\
& = & K(\alpha_1, ... , \alpha_r +1) \cdot \bigl[  K(\alpha_2, ... , \alpha_r)\cdot  K(b_{r+1}, ..., b_m) \bigr. \nonumber \\
&  & +\bigl. K(\alpha_2, ... , \alpha_{r-1}) \cdot K(b_{r+2}, ..., b_m) \bigr] \nonumber \\
&  & - K(\alpha_2, ... , \alpha_{r} + 1)\cdot \bigl[ K(\alpha_1, ... , \alpha_r)\cdot  K(b_{r+1}, ..., b_m) \bigr. \nonumber \\
&  & +\bigl. K(\alpha_1, ... , \alpha_{r-1}) \cdot K(b_{r+2}, ..., b_m) \bigr] \nonumber \\
& = & K(b_{r+1}, ..., b_m) \cdot \bigl[ K(\alpha_1, ... , \alpha_r +1) \cdot K(\alpha_2, ... , \alpha_{r}) \bigr. \nonumber \\
& & \bigl. - K(\alpha_2, ... , \alpha_{r}+1) \cdot K(\alpha_1, ... , \alpha_r) \bigr] \nonumber \\
& & + K(b_{r+2}, ..., b_m) \cdot \bigl[ K(\alpha_1, ... , \alpha_r +1) \cdot K(\alpha_2, ... , \alpha_{r-1}) \bigr. \nonumber \\
& & \bigl. - K(\alpha_2, ... , \alpha_{r}+1) \cdot K(\alpha_1, ... , \alpha_{r-1}) \bigr].
\label{GH_conteq1}
\end{eqnarray}

Using the definition \ref{GH_cont} for $K(\alpha_1, ... , \alpha_r +1)$, we obtain:
\begin{equation}
K(\alpha_1, ... , \alpha_r +1) = K(\alpha_1, ... , \alpha_r) + K(\alpha_1, ... , \alpha_{r-1}),
\label{GH_conteq2}
\end{equation}

and, applying the Eq. \ref{GH_contdet}, we get the following value for the first bracket in Eq. \ref{GH_conteq1}:
\begin{eqnarray*}
K(\alpha_1, ... , \alpha_r +1) \cdot K(\alpha_2, ... , \alpha_{r}) - K(\alpha_2, ... , \alpha_{r}+1) \cdot K(\alpha_1, ... , \alpha_r) \\
= K(\alpha_1, ... , \alpha_r)\cdot K(\alpha_2, ... , \alpha_{r}) + K(\alpha_1, ... , \alpha_{r-1})\cdot K(\alpha_2, ... , \alpha_{r}) \\
- K(\alpha_2, ... , \alpha_r)\cdot K(\alpha_1, ... , \alpha_{r}) - K(\alpha_2, ... , \alpha_{r-1})\cdot K(\alpha_1, ... , \alpha_{r}) \\
 = K(\alpha_1, ... , \alpha_{r-1})\cdot K(\alpha_2, ... , \alpha_{r}) - K(\alpha_2, ... , \alpha_{r-1})\cdot K(\alpha_1, ... , \alpha_{r}) \\
 = (-1)^{r+1} = 1.
\label{GH_conteq3}
\end{eqnarray*}

In addition, for the second bracket in Eq. \ref{GH_conteq1}, we have:

\begin{eqnarray*}
K(\alpha_1, ... , \alpha_r +1) \cdot K(\alpha_2, ... , \alpha_{r-1}) - K(\alpha_2, ... , \alpha_{r}+1) \cdot K(\alpha_1, ... , \alpha_{r-1}) \\
= K(\alpha_1, ... , \alpha_r)\cdot K(\alpha_2, ... , \alpha_{r-1}) + K(\alpha_1, ... , \alpha_{r-1})\cdot K(\alpha_2, ... , \alpha_{r-1}) \\
- K(\alpha_2, ... , \alpha_r)\cdot K(\alpha_1, ... , \alpha_{r-1}) - K(\alpha_2, ... , \alpha_{r-1})\cdot K(\alpha_1, ... , \alpha_{r-1}) \\
 = K(\alpha_1, ... , \alpha_r)\cdot K(\alpha_2, ... , \alpha_{r-1}) -
  K(\alpha_2, ... , \alpha_r)\cdot K(\alpha_1, ... , \alpha_{r-1}) \\
 = (-1)^{r} = -1.
\label{GH_conteq4}
\end{eqnarray*}

Therefore, the Eq. \ref{GH_conteq1} is simplified as follows:
\begin{eqnarray}
q_{2i-1} \cdot p  - q \cdot p_{2i-1} = K(b_{r+1}, ..., b_m) -  K(b_{r+2}, ..., b_m),
\label{GH_conteq5}
\end{eqnarray}
where the right-hand side is:
\begin{eqnarray}
K(b_{r+1}, ..., b_m) -  K(b_{r+2}, ..., b_m) \nonumber \\
= b_{r+1}\cdot K(b_{r+2}, ..., b_m) + K(b_{r+3}, ..., b_m) -  K(b_{r+2}, ..., b_m) \nonumber \\
= (b_{r+1}-1)\cdot K(b_{r+2}, ..., b_m) + K(b_{r+3}, ..., b_m) \nonumber \\
= K(b_{r+1}- 1, b_{r+2}, ..., b_m),
\label{GH_conteq6}
\end{eqnarray}

i.e., it is the denominator of the truncated continued fraction $[b_{r+1} - 1, ... , b_m]$, as we aimed to demonstrate. Observe that if $r$ is even, $p_{2i}/q_{2i}$ would have a continued fraction expression $[\alpha_1, ... , \alpha_r + 1]$, and the argumentation would be valid with the opposite sign $-q_{2i} \cdot p  + q \cdot p_{2i}$. \\

It remains to verify the result for rational numbers: 
$$
\frac{p}{q} \in \left(\frac{a_i}{b_i} = [\alpha_1, ... , \alpha_r], \frac{p_{2i}}{q_{2i}} = [\alpha_1, ... , \alpha_r - 1,2] , \right).
$$
In that case, the first $r$ terms of the continued fraction of $p/q$ are determined as:

 $$\frac{p}{q} = [\alpha_1, ... , \alpha_r - 1,1, b_{r+2}, ... , b_m].$$
 
Reproducing the scheme of the demonstration in the last case, we get

\begin{eqnarray}
-q_{2i} \cdot p  + q \cdot p_{2i} &=& -K(\alpha_1, ... , \alpha_r -1, 2) \cdot K(\alpha_2, ... , \alpha_r-1,1, b_{r+2}, ..., b_m) \nonumber \\
& & +K(\alpha_1, ... , \alpha_r-1,1, b_{r+2}, ..., b_m) \cdot K(\alpha_2, ... , \alpha_r -1,2) \nonumber \\
& = & -K(\alpha_1, ... , \alpha_r -1, 2) \cdot \bigl[  K(\alpha_2, ... , \alpha_r-1,1)\cdot  K(b_{r+2}, ..., b_m) \bigr. \nonumber \\
&  & +\bigl. K(\alpha_2, ... , \alpha_r -1) \cdot K(b_{r+3}, ..., b_m) \bigr] \nonumber \\
&  & + K(\alpha_2, ... , \alpha_r -1,2)\cdot \bigl[ K(\alpha_1, ... , \alpha_r-1,1)\cdot  K(b_{r+2}, ..., b_m) \bigr. \nonumber \\
&  & +\bigl. K(\alpha_1, ... , \alpha_r-1) \cdot K(b_{r+3}, ..., b_m) \bigr] \nonumber \\ 
& = & -K(b_{r+2}, ..., b_m) \cdot \bigl[ K(\alpha_1, ... , \alpha_r -1,2) \cdot K(\alpha_2, ... , \alpha_{r}) \bigr. \nonumber \\
& & \bigl. - K(\alpha_2, ... , \alpha_{r}-1,2) \cdot K(\alpha_1, ... , \alpha_r) \bigr] \nonumber \\
& & + K(b_{r+3}, ..., b_m) \cdot \bigl[ K(\alpha_2, ... , \alpha_r -1,2) \cdot K(\alpha_1, ... , \alpha_{r}-1) \bigr. \nonumber \\
& & \bigl. - K(\alpha_1, ... , \alpha_{r-1},2) \cdot K(\alpha_2, ... , \alpha_{r}-1) \bigr].
\label{GH_conteq7}
\end{eqnarray}

Using the definition \ref{GH_cont} for $K(\alpha_1, ... , \alpha_{r}-1,2)$, we have:

\begin{eqnarray}
K(\alpha_1, ... , \alpha_{r}-1,2) &=& 2\cdot K(\alpha_1, ... , \alpha_{r}-1) + K(\alpha_1, ... , \alpha_{r-1}) \nonumber \\
&=& 2\cdot \left[(\alpha_{r}-1)\cdot K(\alpha_1, ... , \alpha_{r-1}) + K(\alpha_1, ... , \alpha_{r-2})\right]  + K(\alpha_1, ... , \alpha_{r-1})\nonumber \\
&=& 2\cdot(\alpha_{r})\cdot K(\alpha_1, ... , \alpha_{r-1}) - 2\cdot K(\alpha_1, ... , \alpha_{r-1}) \nonumber \\
& &+ 2\cdot K(\alpha_1, ... , \alpha_{r-2}) + K(\alpha_1, ... , \alpha_{r-1}) \nonumber \\
&=& 2\cdot(\alpha_{r})\cdot K(\alpha_1, ... , \alpha_{r-1}) -  K(\alpha_1, ... , \alpha_{r-1}) + 2\cdot K(\alpha_1, ... , \alpha_{r-2}) \nonumber \\
&=& 2\cdot K(\alpha_1, ... , \alpha_{r}) -  K(\alpha_1, ... , \alpha_{r-1}). 
\label{GH_conteq8}
\end{eqnarray}

The first bracket in Eq. \ref{GH_conteq7} is now:
\begin{eqnarray*}
K(\alpha_1, ... , \alpha_r -1,2) \cdot K(\alpha_2, ... , \alpha_{r}) - K(\alpha_2, ... , \alpha_{r}-1,2) \cdot K(\alpha_1, ... , \alpha_r) \\
= 2\cdot K(\alpha_1, ... , \alpha_{r})\cdot K(\alpha_2, ... , \alpha_{r}) -  K(\alpha_1, ... , \alpha_{r-1})\cdot K(\alpha_2, ... , \alpha_{r}) \\
- 2\cdot K(\alpha_2, ... , \alpha_{r})\cdot K(\alpha_1, ... , \alpha_{r}) + K(\alpha_2, ... , \alpha_{r-1})\cdot K(\alpha_1, ... , \alpha_{r}) \\
 = K(\alpha_2, ... , \alpha_{r-1})\cdot K(\alpha_1, ... , \alpha_{r}) -  K(\alpha_1, ... , \alpha_{r-1})\cdot K(\alpha_2, ... , \alpha_{r}) \\
 = (-1)^{r} = -1,
\label{GH_conteq9}
\end{eqnarray*}

In addition, the second bracket requires the following equality:

$$
K(\alpha_1, ... , \alpha_{r}-1) = K(\alpha_1, ... , \alpha_{r}) - K(\alpha_1, ... , \alpha_{r-1}),
$$
resulting therefore that:

\begin{eqnarray*}
K(\alpha_2, ... , \alpha_r -1,2) \cdot K(\alpha_1, ... , \alpha_{r}-1) - K(\alpha_1, ... , \alpha_{r}-1,2) \cdot K(\alpha_2, ... , \alpha_{r}-1)\\
= 2\cdot K(\alpha_2, ... , \alpha_{r})\cdot K(\alpha_1, ... , \alpha_{r}-1) - K(\alpha_2, ... , \alpha_{r-1})\cdot K(\alpha_1, ... , \alpha_{r}-1) \\
- 2\cdot K(\alpha_1, ... , \alpha_{r})\cdot K(\alpha_2, ... , \alpha_{r}-1) + K(\alpha_1, ... , \alpha_{r-1})\cdot K(\alpha_2, ... , \alpha_{r}-1) \\
= 2\cdot K(\alpha_2, ... , \alpha_{r})\cdot K(\alpha_1, ... , \alpha_{r}) 
- 2\cdot K(\alpha_2, ... , \alpha_{r})\cdot K(\alpha_1, ... , \alpha_{r-1}) \\
- K(\alpha_2, ... , \alpha_{r-1}) \cdot K(\alpha_1, ... , \alpha_{r}) 
+ K(\alpha_2, ... , \alpha_{r-1})\cdot K(\alpha_1, ... , \alpha_{r-1}) \\
-2\cdot K(\alpha_1, ... , \alpha_{r})\cdot K(\alpha_2, ... , \alpha_{r}) 
+ 2\cdot K(\alpha_1, ... , \alpha_{r})\cdot K(\alpha_2, ... , \alpha_{r-1}) \\
+ K(\alpha_1, ... , \alpha_{r-1}) \cdot K(\alpha_2, ... , \alpha_{r}) 
- K(\alpha_1, ... , \alpha_{r-1})\cdot K(\alpha_2, ... , \alpha_{r-1}) \\
= K(\alpha_1, ... , \alpha_{r})\cdot K(\alpha_2, ... , \alpha_{r-1}) - K(\alpha_1, ... , \alpha_{r-1}) \cdot K(\alpha_2, ... , \alpha_{r}) \\
= (-1)^r = -1
\label{GH_conteq10}
\end{eqnarray*}

Thus, Eq. \ref{GH_conteq7} is simplified as:
\begin{eqnarray*}
-q_{2i} \cdot p  + q \cdot p_{2i} = K(b_{r+2}, ..., b_m) - K(b_{r+3}, ..., b_m) = K(b_{r+2}- 1, b_{r+3}, ..., b_m),
\label{GH_conteq11}
\end{eqnarray*}

i.e., it is the denominator of the continued fraction $[b_{r+2} - 1, ... , b_m]$, or equivalently, $s^{(r+1)}$ for $k = \sum_{i=1}^{r-1} \alpha_i +(\alpha_r -1) + 1  + 3 = \sum_{i=1}^{r} \alpha_i  + 3$, as we want to demonstrate. To conclude, if $r$ is even, then $p_{2i-1}/q_{2i-1} = [\alpha_1, ... , \alpha_r -1, 2]$, and changing the sign to equality $q_{2i-1} \cdot p  - q \cdot p_{2i-1}$ will result in equivalent reasoning.
\end{proof}
\section*{Acknowledgments}
The author acknowledges funding from Spanish Ministry of Science and Innovation under project M2505 (PID2020-113737GB-I00). I am grateful to Bartolom\'e Luque and Lucas Lacasa for their guidance and advice.

\newpage


\bibliographystyle{ieeetr}
\bibliography{On_a_degree_distribution_of_Haros_graphs_rev.bib}

\end{document}